\documentclass[11pt,a4paper,reqno]{amsart}
\usepackage[margin=25mm,headheight=14pt,footskip=24pt]{geometry}
  \usepackage[utf8]{inputenc}
  \usepackage[T1]{fontenc}
  \usepackage[english]{babel}

\usepackage{amsmath,amssymb,amsthm}
\usepackage{textcase}
\usepackage{microtype}
\usepackage[unicode,hidelinks]{hyperref}
\allowdisplaybreaks[1]

\numberwithin{equation}{section}
\newtheorem{theorem}{Theorem}[section]
\newtheorem{lemma}[theorem]{Lemma}
\newtheorem{proposition}[theorem]{Proposition}
\newtheorem{conjecture}[theorem]{Conjecture}
\newtheorem{corollary}[theorem]{Corollary}
\theoremstyle{definition}
\newtheorem{definition}[theorem]{Definition}
\newtheorem{remark}[theorem]{Remark}

\DeclareMathOperator{\Sing}{Sing}
\DeclareMathOperator{\codim}{codim}
\newcommand{\CC}{\mathbb C}
\newcommand{\PP}{\mathbb P}
\newcommand{\inlinemath}[1]{\mbox{$#1$}}

\title{Existence of weak Landau--Ginzburg 
models for weighted complete intersections}

\author{Victor Przyjalkowski, Constantin Shramov}

\begin{document}

\begin{abstract}
We prove that every smooth well formed Fano or Calabi--Yau weighted complete intersection
admits a nef-partition. As an application, we establish the existence
of weak Landau--Ginzburg models for all such Fano weighted complete intersections.
\end{abstract}

\maketitle

\section{Introduction}

One of the aims of mirror symmetry for Fano varieties is to construct explicit
Landau--Ginzburg models relating the Gromov--Witten invariants of Fano varieties to the
periods of their mirrors. In Givental's construction~\cite{Giv} for complete intersections
in smooth toric varieties, such a model is determined by a nef-partition of the toric
boundary divisors. The notion of a nef-partition was introduced by Borisov~\cite{Bor}
in his construction of conjectural mirror families for Calabi--Yau complete intersections
in toric varieties.

Weighted projective spaces are also toric varieties, but they are usually
singular.
In~\cite{Prz07} (see also~\cite{Wang}), Givental's formula for the Gromov--Witten series was extended to smooth
complete intersections that lie in singular toric varieties and are disjoint from
the singular locus of the ambient variety. These include the smooth well formed weighted complete
intersections considered in this paper.

Let~\inlinemath{X} be a smooth Fano variety of dimension~\inlinemath{n}, and let~\inlinemath{f\in\CC[y_1^{\pm1},\ldots,y_n^{\pm1}]} be a Laurent polynomial in~\inlinemath{n} variables.
Write~\inlinemath{[f]} for the constant term of~\inlinemath{f}, that is,
the coefficient
at~\inlinemath{y_1^0\cdot\ldots\cdot y_n^0}. The series
\[
 \Phi_f(t)=\sum_{r=0}^{\infty}{[f^r]}t^r
\]
is called the \emph{main period} of~\inlinemath{f}. The polynomial~\inlinemath{f} is a
\emph{weak Landau--Ginzburg model} of~\inlinemath{X} if
\[
 \Phi_f(t)=\widetilde I_X(t),
\]
where~\inlinemath{\widetilde I_X(t)} is the regularized generating series for the genus-zero
one-point Gromov--Witten invariants of~\inlinemath{X}. See~\cite[\S2.1 and \S3]{PrzToric}
for details.

In this paper, we consider smooth well formed Fano and Calabi--Yau weighted complete
intersections. For these varieties, the existence problem for nef-partitions can be
stated in terms of the weights of the coordinates
and the degrees of the defining equations.

\begin{definition}[{cf.\ \cite[\S2]{Prz11}}]
Let~\inlinemath{X\subset\PP(a_0,\ldots,a_N)} be a weighted complete intersection
of multidegree~\inlinemath{(d_1,\ldots,d_k)} in a weighted projective space
with weights~\inlinemath{a_0,\ldots,a_N}.
A \emph{nef-partition} for~\inlinemath{X} is a collection of pairwise disjoint subsets~\inlinemath{I_1,\ldots,I_k\subset\{0,\ldots,N\}} such that, for every~\inlinemath{1\mathrel{\leqslant} j\mathrel{\leqslant} k},
\[
 \sum_{i\in I_j}a_i=d_j.
\]
In the Calabi--Yau case, these subsets partition the entire set of indices of the weights.
In the Fano case, write~\inlinemath{I_0} for the residual block~\inlinemath{\{0,\ldots,N\}\setminus(I_1\cup\ldots\cup I_k)}.
A nef-partition for a Fano weighted complete intersection is called \emph{nice} if~\inlinemath{I_0} contains the index of a weight equal to~\inlinemath{1}.
\end{definition}

Given a nice nef-partition for a Fano weighted complete intersection, the change of
variables in~\cite[\S2]{Prz11} transforms Givental's model birationally
into a family defined by a Laurent polynomial. By~\cite[Theorem 0.1]{Giv}
and~\cite[Theorem 1.1(1), Remark 3.2.1(iii)]{Prz07}
(see also~\cite{Wang}), adding a suitable constant to this
polynomial gives a weak Landau--Ginzburg model of the weighted complete intersection.

\begin{conjecture}[{cf.\ \cite[\S2]{Prz11}}]\label{conj:nef}
Every smooth well formed Fano weighted complete intersection
admits a nice nef-partition.
\end{conjecture}

This conjecture was previously known in several special cases. A nice nef-partition
was constructed in~\cite[Proposition 2.2]{Prz11} for smooth Fano complete intersections
of Cartier divisors in weighted projective spaces. Toric degenerations of these
complete intersections, determined by the Newton polytopes of the corresponding
Laurent polynomials, were constructed in~\mbox{\cite[Theorem 2.2]{ILP}}.

The existence of a nice nef-partition for every smooth well formed Fano weighted
complete intersection of codimension~\inlinemath{2} was proved in~\cite[Theorem 1.3]{PSnef}.
Nef-partitions and the corresponding Laurent polynomials for Fano weighted complete
intersections of dimension at most five are listed in~\cite[\S\,T.5]{WCI}.
Further sufficient conditions on the weights for the existence of a nef-partition
were obtained in~\cite[Theorem 1.10]{Ovc}.

\begin{samepage}
To apply Givental's construction to the entire class under consideration,
one has to establish the existence of nice nef-partitions. We prove
Conjecture~\ref{conj:nef} in the following stronger form.

\begin{theorem}\label{thm:main}
Every smooth well formed Fano or Calabi--Yau weighted complete intersection~\inlinemath{X} admits a nef-partition. If~\inlinemath{X} is Fano,
the nef-partition can be chosen so that all weights in the residual block~$I_0$
are equal to~\inlinemath{1}.
\end{theorem}
\end{samepage}

Theorem~\ref{thm:main} allows us to establish the existence of weak Landau--Ginzburg
models satisfying the toric condition for smooth well formed Fano weighted complete
intersections. Recall (see~\mbox{\cite[\S6.1]{Prz13}}) that a weak Landau--Ginzburg model~\inlinemath{f} of a variety~\inlinemath{X} satisfies the \emph{toric condition} if~\inlinemath{X} admits a flat degeneration to a toric variety for which the convex hull
of the primitive integral generators of the rays of its fan is the Newton polytope
of~\inlinemath{f}.

\begin{corollary}\label{cor:landau-ginzburg}
Every smooth well formed Fano weighted complete intersection
has a weak Landau--Ginzburg model satisfying the toric condition.
\end{corollary}

\begin{remark}
By~\cite[Theorem 1.1]{Prince}, the Landau--Ginzburg models arising from different nice
nef-partitions of the same weighted complete intersection are fibrewise birational.
\end{remark}

Recall (see~\cite[Definition 1.6]{PrzStrong}) that a \emph{log Calabi--Yau compactification}
of a Laurent polynomial~\inlinemath{f} in~\inlinemath{n} variables is a normal variety~\inlinemath{Z} containing the algebraic torus~\inlinemath{(\CC^*)^n} as an open subset, together with a proper
morphism~\inlinemath{u\colon Z\to\PP^1} such that the restriction of~\inlinemath{u} to
this torus is given by~\inlinemath{f}, the variety~\inlinemath{Z\setminus u^{-1}(\infty)}
is smooth, and the anticanonical class of~\inlinemath{Z} is linearly equivalent to a
fibre of~\inlinemath{u}. The variety~\inlinemath{Z\setminus u^{-1}(\infty)} itself is
called a \emph{Calabi--Yau compactification} (one may choose an anticanonical divisor
on~\inlinemath{Z} supported in the fibre over infinity).
A weak Landau--Ginzburg model is \emph{toric} if it satisfies the toric condition
and admits a (log) Calabi--Yau compactification.

\begin{samepage}
\begin{conjecture}[{\cite[\S6.1]{Prz13}}]\label{conj:compactification}
Every smooth Fano variety has a toric Landau--Ginzburg model.
\end{conjecture}
\end{samepage}

Thus, Corollary~\ref{cor:landau-ginzburg} shows that, for weighted complete
intersections, proving Conjecture~\ref{conj:compactification} reduces to constructing
a (log) Calabi--Yau compactification. Log Calabi--Yau compactifications were
constructed in~\cite[Theorem 1]{PrzCY} for Fano complete intersections in ordinary
projective spaces, and in~\cite[Theorem 3.1]{PrzStrong} for smooth cyclic covers of
projective spaces when the cover has Fano index~\inlinemath{1}.

Recall that a nef-partition of a Fano weighted complete intersection is \emph{strong}
if all weights in~\inlinemath{I_0} are equal to~\inlinemath{1} and $a_i$ divides $d_j$
whenever~\inlinemath{i\in I_j} and~\inlinemath{1\mathrel{\leqslant} j\mathrel{\leqslant} k}
(see~\cite[Remark 4.2]{PrzStrong}). Ovcharenko~\cite{Ovc} showed that a smooth
well formed weighted complete intersection need not admit a strong nef-partition;
see Remark~\ref{rem:no-strong}.

The idea of the proof of Theorem~\ref{thm:main} is as follows. We first reduce to
the case where the weighted complete intersection~$X$ is not a numerical
intersection with a linear cone. We then need to pack the weights greater than~$1$
into blocks corresponding to the degrees of the defining equations, so that the
sum of the weights in each block does not exceed the corresponding degree.
To do this, we arrange both the weights and the degrees in non-increasing order and fill
the blocks step by step,
assigning to each block the longest initial segment of the
remaining weights that fits into it.

If some weights remain after all blocks have been filled, we show that every
block contains at least two weights. A block containing at most one weight would
contradict the induction hypothesis applied to suitable initial segments of the
weight and degree sequences. Comparing the products of the weights and degrees
then rules out any unassigned weights.
Finally, we add weights equal to~$1$ to each block until its sum equals the
corresponding degree. In the Fano case, at least one weight equal to~$1$
remains;
in the Calabi--Yau case, all weights are used.

\medskip
The paper is organized as follows. In~\S\ref{sec:preliminaries}, we recall the
definitions and preliminary results on weighted complete intersections that we need.
In~\S\ref{sec:arithmetic}, we prove a combinatorial theorem on packing weights into
blocks of prescribed capacities. In~\S\ref{sec:geometry}, we prove the main theorem
and apply it to construct weak Landau--Ginzburg models and degenerations of
Calabi--Yau weighted complete intersections.

\subsection*{Conventions and notation}

We work over the field~\inlinemath{\CC} of complex numbers.

We write~\inlinemath{\#A} for the number of elements of a finite set or finite
collection~\inlinemath{A}; in a collection, repeated values are counted with
multiplicity. The empty sum and the empty product are~\inlinemath{0}
and~\inlinemath{1}, respectively.

\section{Preliminaries}
\label{sec:preliminaries}

We recall the definitions and results on weighted projective spaces and weighted
complete intersections that will be used below. For further details, see~\cite{WCI}.

Let~\inlinemath{a_0,\ldots,a_N} be positive integers. The weighted projective space~\inlinemath{\PP=\PP(a_0,\ldots,a_N)} is defined by
\[
\PP= \operatorname{Proj}\CC[x_0,\ldots,x_N],\qquad \deg x_i=a_i.
\]
A \emph{weighted complete intersection}~\inlinemath{X\subset\PP} of multidegree~\inlinemath{(d_1,\ldots,d_k)} is a subvariety whose homogeneous ideal is generated
by a regular sequence~\inlinemath{f_1,\ldots,f_k} of degrees~\inlinemath{d_1,\ldots,d_k}. In particular,~\inlinemath{\dim X=N-k}.
A smooth projective variety of positive dimension is \emph{Fano} if its anticanonical class is ample,
and \emph{Calabi--Yau} if its canonical class is trivial.

The weighted projective space~\inlinemath{\PP} is \emph{well formed} if the greatest
common divisor of any~\inlinemath{N} of the~$N+1$ weights $a_0,\ldots,a_N$
is~\inlinemath{1}. A subvariety~\inlinemath{X\subset\PP} is well formed if $\PP$
is well formed and
\[
 \codim_X\bigl(X\cap\Sing\PP\bigr)\mathrel{\geqslant}2.
\]
A weighted complete intersection~\inlinemath{X} is a \emph{numerical intersection
with a linear cone} if~\inlinemath{d_j=a_i} for some~\inlinemath{i,j}.

\begin{proposition}[{\cite[Proposition 5.4.1]{WCI}}]\label{prop:fano-cy}
Let~\inlinemath{X\subset\PP(a_0,\ldots,a_N)} be a smooth well formed weighted
complete intersection of multidegree~\inlinemath{(d_1,\ldots,d_k)} and positive
dimension. Then~\inlinemath{X} is Fano if and only if
\[
\sum_{i=0}^N a_i>\sum_{j=1}^k d_j,
\]
and Calabi--Yau if and only if
\[
\sum_{i=0}^N a_i=\sum_{j=1}^k d_j.
\]
\end{proposition}

The next proposition follows by the same argument as in the proof
of~\cite[Proposition~3.5.6]{WCI}, using~\cite[Theorem 3.7.3]{WCI};
cf.\ also~\cite[Proposition~2.17]{Ovc}.

\begin{proposition}\label{prop:degree-combinations}
Let $X\subset\PP=\PP(a_0,\ldots,a_N)$ be a smooth well formed weighted complete
intersection of multidegree~\inlinemath{(d_1,\ldots,d_k)}. For each nonempty~\inlinemath{I\subset\{0,\ldots,N\}} with~\inlinemath{\gcd(a_i\mid i\in I)>1},
there are at least~\inlinemath{\#I} distinct indices~\inlinemath{j} such that the
degree~\inlinemath{d_j} is a nonnegative integral linear combination of the
weights~\inlinemath{a_i},~\inlinemath{i\in I}.
\end{proposition}

\begin{proof}
Let~\inlinemath{f_1,\ldots,f_k} be homogeneous polynomials of degrees~\inlinemath{d_1,\ldots,d_k} defining~\inlinemath{X}. Consider the coordinate subspace
\[
 \PP_I=\{(x_0:\ldots:x_N)\mid x_i=0 \text{\ for all\ } i\notin I\}\subset\PP.
\]
According to~\cite[Theorem 1.3.8]{WCI},
the condition~\inlinemath{\gcd(a_i\mid i\in I)>1}
implies that~\inlinemath{\PP_I\subset\Sing\PP}.
By~\mbox{\cite[Proposition 3.7.1]{WCI}}, the smooth well formed weighted complete
intersection~\inlinemath{X} is disjoint from the singular locus of the weighted
projective space~$\PP$. Hence~\mbox{$X\cap\PP_I=\varnothing$}.

Since~\inlinemath{\dim\PP_I=\#I-1}, any~\inlinemath{\#I-1} nonzero homogeneous
polynomials have a common zero on~\inlinemath{\PP_I}. Thus the restrictions
of~\inlinemath{f_j} to $\PP_I$ are not identically zero for at least~\inlinemath{\#I}
distinct indices~\inlinemath{j}. Each such restriction contains a monomial in the
variables~\inlinemath{x_i},~\inlinemath{i\in I}. Its weighted degree is~\inlinemath{d_j}, giving the required expression for~\inlinemath{d_j}.
\end{proof}

\section{Weights and degrees}
\label{sec:arithmetic}

In this section, we prove the combinatorial results needed for the proof
of Theorem~\ref{thm:main}.

Let~\inlinemath{A=(a_1,\ldots,a_s)} and~\inlinemath{D=(d_1,\ldots,d_k)} be finite
collections of integers, where~\inlinemath{a_i\mathrel{\geqslant}2}
and~\inlinemath{d_j\mathrel{\geqslant}2}; either collection may be empty.
As in~\S\ref{sec:preliminaries}, we call the numbers $a_i$ and $d_j$ weights
and degrees, respectively. A \emph{packing} of~\inlinemath{A} into blocks of
capacities~\inlinemath{d_1,\ldots,d_k} is a partition of the index set~\inlinemath{\{1,\ldots,s\}} into pairwise disjoint blocks~\inlinemath{B_1,\ldots,B_k} such that, for every~\inlinemath{1\mathrel{\leqslant} j\mathrel{\leqslant} k},
\[
 \sum_{i\in B_j}a_i\mathrel{\leqslant} d_j.
\]
Blocks are allowed to be empty.

Our aim is to prove the following theorem. Under the condition on weights and
degrees imposed by Proposition~\ref{prop:degree-combinations}, it says that all
weights greater than one can be packed into $k$ blocks so that the sum of the
weights in each block does not exceed the corresponding degree. Moreover,
a simple greedy algorithm produces such a packing.

\begin{theorem}\label{thm:numerical}
Suppose that~\inlinemath{a_i\ne d_j} for all $i$ and $j$. Assume that, for every
nonempty subset~\inlinemath{I\subset\{1,\ldots,s\}} with~\inlinemath{\gcd(a_i\mid i\in I)>1}, there are at least~\inlinemath{\#I} distinct
indices~\inlinemath{j} such that~\inlinemath{d_j} is a nonnegative integral linear
combination of the weights~\inlinemath{a_i},~\inlinemath{i\in I}.
Arrange both the weights and the degrees in non-increasing order. For~\inlinemath{j=1,\ldots,k}, step by step,
fill the $j$-th block with the longest initial
segment of the remaining sequence of weights that has sum at most~\inlinemath{d_j}. Then this algorithm assigns every weight to a block,
and hence gives a packing of~\inlinemath{A} into blocks of
capacities~\mbox{$d_1,\ldots,d_k$}.
\end{theorem}

We need two lemmas to prove Theorem~\ref{thm:numerical}.
For a positive integer~\inlinemath{n}, set
\[
 H(n)=
 \begin{cases}
 n,&\text{if $n$ is odd},\\[2pt]
 \frac{9n}{8},&\text{if $n$ is even}.
 \end{cases}
\]
\begin{lemma}\label{lem:modified-product}
Let~\inlinemath{r\mathrel{\geqslant}2} and~\inlinemath{c_1,\ldots,c_r\mathrel{\geqslant}2} be integers. Suppose that a positive
integer~\inlinemath{d} satisfies
\[
 d<\sum_{i=1}^r c_i+\min_{1\mathrel{\leqslant} i\mathrel{\leqslant} r}c_i.
\]
Then
\[
 H(d)\mathrel{\leqslant}\prod_{i=1}^r H(c_i).
\]
\end{lemma}

\begin{proof}
For~\inlinemath{r=2}, set~\inlinemath{u=\max(c_1,c_2)}
and~\inlinemath{v=\min(c_1,c_2)}. Then
\[
d\mathrel{\leqslant} u+2v-1.
\]
If~\inlinemath{u=v=2}, then~\inlinemath{d\mathrel{\leqslant}5} and
\[
 H(d)\mathrel{\leqslant}5<\frac{81}{16}=H(2)^2=H(u)H(v).
\]
If~\inlinemath{v=2} and~\inlinemath{u\mathrel{\geqslant}3}, then
\begin{equation}\label{eq:HuHv-1}
 H(d)\mathrel{\leqslant}\frac{9d}{8}\mathrel{\leqslant}\frac{9(u+3)}{8}\mathrel{\leqslant}\frac{9u}{4}\mathrel{\leqslant} H(u)H(2)=H(u)H(v).
\end{equation}
Finally, if~\inlinemath{u\mathrel{\geqslant} v\mathrel{\geqslant}3}, we have
\[
 8uv-9(u+2v-1)=(u-3)(8v-9)+6(v-3)\mathrel{\geqslant}0,
\]
so
\begin{equation}\label{eq:HuHv-2}
 H(d)\mathrel{\leqslant}\frac{9}{8}\big(u+2v-1\big)\mathrel{\leqslant} uv\mathrel{\leqslant} H(u)H(v).
\end{equation}
This proves the lemma for~\inlinemath{r=2}. For~\inlinemath{d=u+v}, the resulting
inequality is strict: since
\[
u+v<u+2v-1,
\]
replacing $u+2v-1$ by $u+v$ in the estimates~\eqref{eq:HuHv-1}
and~\eqref{eq:HuHv-2} gives a strict inequality for~$H(d)$.
Thus, for all integers~\inlinemath{u,v\mathrel{\geqslant}2}, we have
\begin{equation}\label{eq:merge-product}
 H(u+v)<H(u)H(v).
\end{equation}

For~\inlinemath{r\mathrel{\geqslant}3}, order the numbers so that~\inlinemath{c_1\mathrel{\geqslant}\ldots\mathrel{\geqslant} c_r}, and replace the
two largest numbers by their sum. This leaves the sum and the minimum of the collection
unchanged. Induction on~\inlinemath{r}, together with~\eqref{eq:merge-product}, gives
\[
 H(d)\mathrel{\leqslant} H(c_1+c_2)\prod_{i=3}^r H(c_i)
 <\prod_{i=1}^r H(c_i).\qedhere
\]
\end{proof}

\begin{lemma}[{cf.\ \cite[Proposition 6.3.2]{WCI}}]\label{lem:divisor-products}
Suppose that for each nonempty~\inlinemath{I\subset\{1,\ldots,s\}} with~\inlinemath{\gcd(a_i\mid i\in I)>1}, there are at least~\inlinemath{\#I} distinct
indices~\inlinemath{j} such that~\inlinemath{d_j} is a nonnegative integral linear
combination of the weights~\inlinemath{a_i},~\inlinemath{i\in I}.
Then, for every integer~\inlinemath{q\mathrel{\geqslant}2}, one has
\begin{equation}\label{eq:divisor-counts}
 \#\{i\mid a_i\text{ is divisible by }q\}
 \mathrel{\leqslant}\#\{j\mid d_j\text{ is divisible by }q\},
\end{equation}
and
\begin{equation}\label{eq:prod-weights-vs-degrees}
 \prod_{i=1}^s a_i\mathrel{\leqslant}\prod_{j=1}^k d_j.
\end{equation}
\end{lemma}

\begin{proof}
The set of indices of weights divisible by~\inlinemath{q} is either empty or
satisfies the common-divisor condition in the hypothesis. Every integral linear
combination of these weights is also divisible by~\inlinemath{q}.
Thus~\eqref{eq:divisor-counts} holds in either case.

For a prime~\inlinemath{p}, let~\inlinemath{\nu_p(n)} denote the exponent
of~\inlinemath{p} in the prime factorization of a positive integer~\inlinemath{n}.
This is the number of positive integers~\inlinemath{r} for which~\inlinemath{n}
is divisible by~\inlinemath{p^r}. Summing~\eqref{eq:divisor-counts}
with~\inlinemath{q=p^r} over all positive integers~\inlinemath{r}, we obtain
\[
 \begin{aligned}
 \sum_{i=1}^s\nu_p(a_i)
 &=\sum_{r=1}^{\infty}\#\{i\mid a_i\text{ is divisible by }p^r\}\\
 &\mathrel{\leqslant}\sum_{r=1}^{\infty}\#\{j\mid d_j\text{ is divisible by }p^r\}
 =\sum_{j=1}^k\nu_p(d_j).
 \end{aligned}
\]
Thus the exponent of each prime~\inlinemath{p} in the product of the weights
does not exceed its exponent in the product of the degrees. Hence the product
of the weights divides the product of the degrees, which implies
inequality~\eqref{eq:prod-weights-vs-degrees}.
\end{proof}

We can now prove Theorem~\ref{thm:numerical}.

\begin{proof}[Proof of Theorem~\ref{thm:numerical}]
We argue by induction on the number~\inlinemath{k} of degrees. If~\inlinemath{k=0},
then $D$, and hence~\inlinemath{A}, is empty, so there is nothing to prove.

Suppose that~\inlinemath{k\mathrel{\geqslant}1} and that the assertion holds for
fewer than~\inlinemath{k} degrees. If~\inlinemath{A} is empty, there is again
nothing to prove. Relabel the weights and degrees so that
\[
 a_1\mathrel{\geqslant}\ldots\mathrel{\geqslant} a_s\mathrel{\geqslant}2,\qquad d_1\mathrel{\geqslant}\ldots\mathrel{\geqslant} d_k,
\]
and pack the weights using the algorithm in the statement of the theorem.
Suppose that a weight remains unassigned after all~\inlinemath{k} blocks
have been filled.

We first show that every block must then contain at least two weights.
Suppose that the $t$-th block contains at most one weight, and that the first~\inlinemath{t-1} blocks contain~\inlinemath{m} weights in total.
If the $t$-th block is empty, then~\inlinemath{d_t<a_{m+1}}. If it contains one
weight, then the next weight~\inlinemath{a_{m+2}} exists and does not fit into
this block, so
\[
d_t<a_{m+1}+a_{m+2}\mathrel{\leqslant}2a_{m+1}.
\]
In either case, we have~\inlinemath{d_t<2a_{m+1}}.

Consider the collections
\[
 A'=(a_1,\ldots,a_{m+1}),\qquad D'=(d_1,\ldots,d_{t-1}).
\]
We check that they satisfy the hypotheses of the theorem. Let~\inlinemath{I\subset\{1,\ldots,m+1\}} be a nonempty subset with~\inlinemath{\gcd(a_i\mid i\in I)>1}. By assumption, there are at least~\inlinemath{\#I} distinct indices~\inlinemath{j} for which~\inlinemath{d_j}
is a nonnegative integral linear combination of the weights~\inlinemath{a_i},~\inlinemath{i\in I}. Each such expression has at least two
summands, counted with multiplicity, since~\inlinemath{d_j\ne a_i} for every~\inlinemath{i}. All the weights involved
are at least~\inlinemath{a_{m+1}}, so these degrees $d_j$ are at least~\inlinemath{2a_{m+1}>d_t} and therefore belong to~\inlinemath{D'}.
The inequalities~\inlinemath{a_i\ne d_j} also hold for $A'$ and $D'$.
Thus~\inlinemath{A'} and~\inlinemath{D'} satisfy all the hypotheses of the theorem.

By the induction hypothesis, the algorithm must assign all~\inlinemath{m+1}
weights of~\inlinemath{A'} to the first~\inlinemath{t-1} blocks. However,
restricting to these collections leaves the first~\inlinemath{t-1} steps unchanged: all weights
assigned at these steps, as well as the first weight that fails to fit at each
step, are retained in~\inlinemath{A'}. These blocks therefore still contain
exactly~\inlinemath{m} weights. This contradiction proves that every block
contains at least two weights.

Let~\inlinemath{M} be the number of weights assigned to blocks after
all~\inlinemath{k} steps of the algorithm. By assumption,
we have~\inlinemath{M<s}.
Write~\inlinemath{B_j} for the set of indices of the weights assigned to the
$j$-th block, and~\inlinemath{w_j} for the first weight still unassigned after
that block has been filled. The packing rule and the ordering of the weights give
\[
 d_j<\sum_{i\in B_j}a_i+w_j
 \mathrel{\leqslant}\sum_{i\in B_j}a_i+\min_{i\in B_j}a_i.
\]
Since~\inlinemath{\#B_j\mathrel{\geqslant}2}, Lemma~\ref{lem:modified-product} gives
\[
 H(d_j)\mathrel{\leqslant}\prod_{i\in B_j}H(a_i).
\]
Multiplying these inequalities over~\inlinemath{j=1,\ldots,k}, and using the
facts that the blocks~\inlinemath{B_j} partition~\inlinemath{\{1,\ldots,M\}}
and that~\inlinemath{H(a_i)>1} for every weight, we obtain
\begin{equation}\label{eq:failed-product}
 \prod_{j=1}^k H(d_j)
 \mathrel{\leqslant}\prod_{i=1}^M H(a_i)
 <\prod_{i=1}^s H(a_i).
\end{equation}

On the other hand, Lemma~\ref{lem:divisor-products} gives
\[
\prod_{i=1}^s a_i\mathrel{\leqslant}\prod_{j=1}^k d_j.
\]
Let~\inlinemath{e_A} and~\inlinemath{e_D} be the numbers of even entries
in~\inlinemath{A} and~\inlinemath{D}, respectively.
Inequality~\eqref{eq:divisor-counts} with~\inlinemath{q=2} gives~\inlinemath{e_A\mathrel{\leqslant} e_D}. Hence
\[
 \prod_{i=1}^s H(a_i)
 =\left(\frac{9}{8}\right)^{e_A}\prod_{i=1}^s a_i
 \mathrel{\leqslant}\left(\frac{9}{8}\right)^{e_D}\prod_{j=1}^k d_j
 =\prod_{j=1}^k H(d_j).
\]
This contradicts~\eqref{eq:failed-product}. Thus every weight is assigned to
a block, completing the proof.
\end{proof}

%

\section{Proof of Theorem~\ref{thm:main} and Corollary~\ref{cor:landau-ginzburg}}\label{sec:geometry}

We now prove the main results of the paper.

\begin{proof}[Proof of Theorem~\ref{thm:main}]
A nef-partition depends only on the weights and degrees.
According to~\mbox{\cite[Corollaries 3.7.4, 3.9.8]{WCI}},
we may therefore replace~\inlinemath{X}
by a general weighted complete intersection of the same multidegree and then
successively remove pairs consisting of a weight and an equal degree. The resulting
weighted complete intersection is smooth and well formed and has the same dimension.
The difference between the sum of the weights and the sum of the degrees is unchanged.
If a weight~\inlinemath{a_i} and a degree~\inlinemath{d_j=a_i} are removed at one step,
we extend a nef-partition of the remaining weights by adding the singleton block~\inlinemath{\{i\}} corresponding to~\inlinemath{d_j}. The sums of the weights in all
other blocks are unchanged, as is the residual block~\inlinemath{I_0} in the Fano
case. Restoring the removed pairs in reverse order gives a nef-partition for the
original weighted complete intersection. It is therefore enough to consider a
weighted complete intersection that is not a numerical intersection with a linear cone.

If~\inlinemath{k=0}, all weights are equal to~\inlinemath{1}; see, for example,
\cite[Theorem 1.3.8]{WCI}. The assertion is then immediate.
Suppose that~\inlinemath{k\mathrel{\geqslant}1}. Since $X$ is not a numerical intersection with a linear cone, one has $d_j\geqslant 2$
for all $1\leqslant j\leqslant k$.
By Proposition~\ref{prop:degree-combinations} and Theorem~\ref{thm:numerical}, we can
pack all weights greater than $1$ into blocks of capacities~\inlinemath{d_1,\ldots,d_k}. Proposition~\ref{prop:fano-cy} gives
\[
\sum_{i=0}^{N}a_i\mathrel{\geqslant}\sum_{j=1}^{k}d_j,
\]
with equality if and only if $X$ is Calabi--Yau. After the weights greater than~$1$
have been assigned to blocks, there are therefore enough weights equal to~$1$
to bring the sum in each block up to~\inlinemath{d_j}.
If $X$ is Calabi--Yau, this uses all weights. If $X$ is Fano, some weights equal
to~$1$ remain, and we use them to form the block~$I_0$. In either case, let~\inlinemath{I_j} be the set of indices of the weights in the $j$-th block.
This gives the required nef-partition.
\end{proof}

\begin{remark}[{\cite[Example 1.7]{Ovc}}]\label{rem:no-strong}
A strong nef-partition need not exist under the hypotheses of Theorem~\ref{thm:main}.
Indeed, for an integer~\inlinemath{t\mathrel{\geqslant}1}, consider a general weighted
complete intersection
\[
 X\subset\PP(1^{(61+t)},6,10,15)
\]
of multidegree $(16,21,25,30)$, where~\inlinemath{1^{(61+t)}} denotes~\inlinemath{61+t} weights equal to one. This is a smooth well formed Fano variety.
Each of the weights~\inlinemath{6,10,15} divides only the degree~\inlinemath{30},
but~\inlinemath{6+10+15=31>30}, so no strong nef-partition exists.
\end{remark}

Next, we use Theorem~\ref{thm:main} to prove the corollary.

\begin{proof}[Proof of Corollary~\ref{cor:landau-ginzburg}]
Let~\inlinemath{X\subset\PP=\PP(a_0,\ldots,a_N)} be a weighted complete intersection
of multidegree~\inlinemath{(d_1,\ldots,d_k)}.
By Theorem~\ref{thm:main}, it admits a nice nef-partition.
Choose any nice nef-partition and set
\[
 i_X=\sum_{i=0}^N a_i-\sum_{j=1}^k d_j>0.
\]
We use the change of variables from~\cite[\S2]{Prz11}. Recall that Givental's model
is the subvariety of the torus with coordinates~\inlinemath{u_0,\ldots,u_N} defined by
\begin{equation}\label{eq:LG-equations}
 \prod_{i=0}^N u_i^{a_i}=1,\qquad
 \sum_{i\in I_j}u_i=1,
\end{equation}
where $1\mathrel{\leqslant} j\mathrel{\leqslant} k$, equipped with the function~\inlinemath{\sum_{i\in I_0}u_i}.
Choose an index~\inlinemath{p_0\in I_0} with~\inlinemath{a_{p_0}=1}, and choose an
arbitrary index~\inlinemath{p_j} in each block~\inlinemath{I_j} for~\inlinemath{1\leqslant j\leqslant k}.
For~\inlinemath{j=0,\ldots,k}, set~\inlinemath{J_j=I_j\setminus\{p_j\}}.
Introduce variables~\inlinemath{y_i} for~\inlinemath{i\in J_0\cup\ldots\cup J_k}.
There are~\inlinemath{N-k=\dim X} such variables. For~\inlinemath{1\mathrel{\leqslant} j\mathrel{\leqslant} k} and~\inlinemath{i\in J_j}, set
\[
 u_{p_j}=\frac{1}{1+\sum_{i\in J_j}y_i},\qquad
 u_i=\frac{y_i}{1+\sum_{s\in J_j}y_s},
\]
and for~\inlinemath{i\in J_0}, set~\inlinemath{u_i=y_i}.
The linear relations in~\eqref{eq:LG-equations} hold identically.
Since~\inlinemath{a_{p_0}=1}, the first relation in~\eqref{eq:LG-equations}
uniquely determines~\inlinemath{u_{p_0}} in terms of the variables~\inlinemath{y_i}.
The inverse change of variables is~\inlinemath{y_i=u_i/u_{p_j}} for~\inlinemath{i\in J_j},~\inlinemath{j\mathrel{\geqslant}1}, and~\inlinemath{y_i=u_i} for~\inlinemath{i\in J_0}.
Substituting into the function~\inlinemath{\sum_{i\in I_0}u_i}, we see that
Givental's model is birational to the family defined by the Laurent polynomial
\begin{equation}\label{eq:laurent-model}
 f=\frac{\displaystyle\prod_{j=1}^k\left(1+\sum_{i\in J_j}y_i\right)^{d_j}}
 {\displaystyle\prod_{j=0}^k\prod_{i\in J_j}y_i^{a_i}}
 +\sum_{i\in J_0}y_i.
\end{equation}

We compute the constant term~\inlinemath{[f^r]} of~\inlinemath{f^r}.
If the first summand in~\eqref{eq:laurent-model} is chosen~\inlinemath{s} times
in the expansion of~\inlinemath{f^r}, then each summand~\inlinemath{y_i},~\inlinemath{i\in J_0}, must be chosen exactly~\inlinemath{a_i s} times to obtain
a constant term. Hence
\[
r=s+\sum_{i\in J_0}a_i s=i_X s.
\]
There are~\inlinemath{(i_X s)!/\bigl(s!\prod_{i\in J_0}(a_i s)!\bigr)} ways to
choose these summands.

In each factor~\inlinemath{\left(1+\sum_{i\in J_j}y_i\right)^{d_j s}}, we must select the monomial~\inlinemath{\prod_{i\in J_j}y_i^{a_i s}}; the summand~\inlinemath{1} is then
chosen~\inlinemath{a_{p_j} s} times. The coefficient of this monomial is~\inlinemath{(d_j s)!/\prod_{i\in I_j}(a_i s)!}.
Multiplying these coefficients and using~\inlinemath{a_{p_0}=1}, we obtain, for
every integer~\inlinemath{s\mathrel{\geqslant}0}, the equality
\[
 [f^{i_X s}]=\frac{(i_X s)!\prod_{j=1}^k(d_j s)!}
 {\prod_{i=0}^N(a_i s)!},
\]
while~\inlinemath{[f^r]=0} for every integer~\inlinemath{r\mathrel{\geqslant}0}
not divisible by~\inlinemath{i_X}.

Thus the main period of~\inlinemath{f} is
\[
 \Phi_f(t)=\sum_{r=0}^{\infty}{[f^r]}t^r
 =\sum_{s=0}^{\infty}
 \frac{(i_X s)!\prod_{j=1}^k(d_j s)!}{\prod_{i=0}^N(a_i s)!}\,t^{i_X s}.
\]
This series is independent of the choice of a nice nef-partition.

By Givental's formula~\cite[Theorem 0.1]{Giv} and its
generalization~\cite[Theorem 1.1]{Prz07}
(see also~\cite[Theorem~1.1 and equation~(3.9)]{Wang}), a suitable choice of the constant term
of~\inlinemath{f} gives
\[
 \Phi_f(t)=\widetilde I_X(t).
\]
Therefore,~\inlinemath{f} is a weak Landau--Ginzburg model of~\inlinemath{X}.

Let~\inlinemath{x_0,\ldots,x_N} be homogeneous coordinates
on~\inlinemath{\PP(a_0,\ldots,a_N)}, and let~\inlinemath{T} be the subvariety
defined, for~\inlinemath{j=1,\ldots,k}, by
\[
 \prod_{i\in I_j}x_i=x_{p_0}^{d_j}.
\]
It was proved in~\cite[Theorem 2.2]{ILP} (see also~\cite[Theorem 9.2.1]{WCI})
that~\inlinemath{T} is toric, that~\inlinemath{X} admits a flat degeneration
to~\inlinemath{T}, and that the fan polytope of~\inlinemath{T} is the Newton
polytope of~\inlinemath{f}. Thus~\inlinemath{f} satisfies the toric condition.
\end{proof}

\begin{corollary}\label{cor:cy-degeneration}
Let~\inlinemath{X\subset\PP=\PP(a_0,\ldots,a_N)} be a smooth well formed
Calabi--Yau weighted complete intersection of multidegree~\inlinemath{(d_1,\ldots,d_k)}, and let~\inlinemath{I_1,\ldots,I_k} be a nef-partition for~\inlinemath{X}.
Then~\inlinemath{X} admits a flat degeneration to the reduced union of
coordinate subspaces defined in homogeneous coordinates~\inlinemath{x_0,\ldots,x_N} by
\[
 \prod_{i\in I_1}x_i=\ldots=\prod_{i\in I_k}x_i=0.
\]
Its irreducible components are obtained by choosing one index~\inlinemath{i_j\in I_j} for each~\inlinemath{j} and are defined in~$\PP$ by
\[
x_{i_1}=\ldots=x_{i_k}=0.
\]
\end{corollary}

\begin{proof}
Let~\inlinemath{f_1,\ldots,f_k} be homogeneous polynomials of degrees~\inlinemath{d_1,\ldots,d_k} defining~\inlinemath{X}, and set
$$
m_j=\prod_{i\in I_j}x_i.
$$
The monomials~\inlinemath{m_j} have degrees~\inlinemath{d_j} and involve pairwise
disjoint sets of variables, so they form a regular sequence.
For a parameter~\inlinemath{t\in\mathbb A^1}, consider the family defined, for~\inlinemath{j=1,\ldots,k}, by
\[
 tf_j+(1-t)m_j=0.
\]
The fibre at~\inlinemath{t=1} is~\inlinemath{X}, while the fibre
at~\inlinemath{t=0} is defined by
\[
m_1=\ldots=m_k=0.
\]
By~\cite[Corollary 2.5.2]{WCI}, there is an open subset~\inlinemath{U\subset\mathbb A^1} containing~\inlinemath{0} and~\inlinemath{1}
over which the polynomials defining the family form a regular sequence.
The dimensions of all graded pieces of the coordinate rings of the fibres
over~\inlinemath{U} are constant, since the Poincar\'e series of each coordinate ring
depends only on the weights and degrees (see~\cite[Corollary A.2.16]{WCI}).
Since~\inlinemath{U} is reduced, the corresponding coherent~\inlinemath{\mathcal O_U}-modules are locally free
(see, for example,~\mbox{\cite[Exercise II.5.8(c)]{Hartshorne}}). Hence the family is flat.
\end{proof}

\begin{samepage}
\medskip
\noindent\textbf{Use of AI.}
The authors used ChatGPT-6 Astra to prove Theorem~\ref{thm:numerical}.
In particular, the AI system proposed the statement and proof of
Lemma~\ref{lem:modified-product} in place of a longer case by case analysis.
AI was also used in preparing the manuscript. The authors independently
carried out the final mathematical verification of the proofs and wrote the final version of the article.
\par
\end{samepage}

\begingroup
\small
\begin{flushleft}
\begin{samepage}
Victor Przyjalkowski\par
Steklov Mathematical Institute, Russian Academy of Sciences,
8 Gubkina Street, Moscow 119333, Russia\par
\texttt{victorprz@mi-ras.ru, victorprz@gmail.com}\par
\end{samepage}
\medskip
\begin{samepage}
Constantin Shramov\par
Steklov Mathematical Institute, Russian Academy of Sciences,
8 Gubkina Street, Moscow 119333, Russia\par
National Research University Higher School of Economics,
6 Usacheva Street, Moscow 119048, Russia\par
\texttt{costya.shramov@gmail.com}\par
\end{samepage}
\end{flushleft}
\endgroup
\end{document}